\documentclass[12pt]{amsart}
\usepackage{hyperref, amssymb}
\usepackage[capitalise, nameinlink]{cleveref}
\usepackage[all]{xy}

\makeatletter
\@namedef{subjclassname@2010}{%
  \textup{2010} Mathematics Subject Classification}
\makeatother
\newtheorem{theorem}{Theorem}[section]
\newtheorem{prop}[theorem]{Proposition}
\newtheorem{cor}[theorem]{Corollary}
\newtheorem{lem}[theorem]{Lemma}

\newtheorem{mainthm}{Theorem}

\theoremstyle{definition}
\newtheorem{defn}[theorem]{Definition}
\newtheorem{rem}[theorem]{Remark}
\newtheorem{ex}[theorem]{Example}

\def\R{{\mathbb{R}}}
\def\N{{\mathbb{N}}}
\def\C{{\mathbb{C}}}
\def\Q{{\mathbb{Q}}}
\def\Z{{\mathbb{Z}}}

\newcommand{\U}{\widehat{\mathcal{U}}}

\hypersetup{
	colorlinks,
	linkcolor=blue,          
	filecolor=magenta,      
	urlcolor=cyan           
}

\begin{document}
\title[Rational K-stability of continuous $C(X)$-algebras]{Rational K-stability of continuous $C(X)$-algebras}
\author[Apurva Seth, Prahlad Vaidyanathan]{Apurva Seth, Prahlad Vaidyanathan}
\address{Department of Mathematics\\ Indian Institute of Science Education and Research Bhopal\\ Bhopal ByPass Road, Bhauri, Bhopal 462066\\ Madhya Pradesh. India.}
\email{apurva17@iiserb.ac.in, prahlad@iiserb.ac.in}
\date{}
\maketitle
\parindent 0pt

\begin{abstract}
We show that the property of being rationally $K$-stable passes from the fibers of a continuous $C(X)$-algebra to the ambient algebra, under the assumption that the underlying space $X$ is compact, metrizable, and of finite covering dimension. As an application, we show that a crossed product C*-algebra is (rationally) K-stable provided the underlying C*-algebra is (rationally) K-stable, and the action has finite Rokhlin dimension with commuting towers.
\end{abstract}


Given a compact Hausdorff space $X$, a continuous $C(X)$-algebra is the section algebra of a continuous field of C*-algebras over $X$. Such algebras form an important class of non-simple C*-algebras, and it is often of interest to understand those properties of a C*-algebra which pass from the fibers to the ambient $C(X)$-algebra. \\

Given a unital C*-algebra, we write $\mathcal{U}_n(A)$ for the group of $n\times n$ unitary matrices over $A$. This is a topological group, and its homotopy groups $\pi_j(\mathcal{U}_n(A))$ are termed the \emph{nonstable $K$-theory} groups of $A$. These groups were first systematically studied by Rieffel \cite{rieffel2} in the context of noncommutative tori. Thomsen \cite{thomsen} built on this work, and developed the notion of quasi-unitaries, thus constructing a homology theory for (possibly non-unital) C*-algebras. \\

Unfortunately, the nonstable K-theory for a given C*-algebra is notoriously difficult to compute explicitly. Even for the algebra of complex numbers, these groups are naturally related to the homotopy groups of spheres $\pi_j(S^n)$, which are not known for many values of $j$ and $n$. It is here that rational homotopy theory has proved to be useful to topologists and, in this paper, we employ this tool in the context of C*-algebras. \\

A C*-algebra $A$ is said to be \emph{$K$-stable} if the homotopy groups $\pi_j(\mathcal{U}_n(A))$ are naturally isomorphic to the $K$-theory groups $K_{j+1}(A)$, and \emph{rationally} K-stable if the analogous statement holds for the rational homotopy groups (see \cref{defn:k_stable}). In \cite{vaidyanathan}, we proved that, for a continuous $C(X)$-algebras, the property of being $K$-stable passes from the fibers to the whole algebra, provided the underlying space $X$ is metrizable and has finite covering dimension. The goal of this paper is to prove an analogous result for rational $K$-stability.

\begin{mainthm}\label{main_thm}
Let $X$ be a compact metric space of finite covering dimension and let $A$ be a continuous $C(X)$-algebra. If each fibre of $A$ is rationally $K$-stable, then so is $A$.
\end{mainthm}

As an interesting application of these results, we consider crossed product C*-algebras where the action has finite Rokhlin dimension (with commuting towers). A theorem of Gardella, Hirshberg and Santiago \cite{gardella} states that such a crossed product C*-algebra can be locally approximated by a continuous $C(X)$-algebra (see \cref{defn: seq_split}). This leads to the following result.

\begin{mainthm}\label{main_thm: rokhlin}
Let $\alpha:G\to \text{Aut}(A)$ be an action of a compact Lie group on a separable C*-algebra $A$ such that $\alpha$ has finite Rokhlin dimension with commuting towers. If $A$ is rationally $K$-stable ($K$-stable), then so is $A\rtimes_{\alpha} G$.
\end{mainthm}

The paper is organized as follows: In \cref{sec: preliminaries}, we introduce the basic notions used throughout the paper - that of nonstable $K$-groups, $C(X)$-algebras, and the rationalization of $H$-spaces. In \cref{sec: main}, we prove \cref{main_thm} along with some applications and examples. Finally, \cref{sec: rokhlin} is devoted to the proof of \cref{main_thm: rokhlin}. 

\section{Preliminaries}\label{sec: preliminaries}

\subsection{Nonstable $K$-theory}\label{subsec:nonstable_k_theory}
We begin by reviewing the work of Thomsen of constructing the nonstable $K$-groups associated to a C*-algebra. For the proofs of the results mentioned in this section, the reader is referred to \cite{thomsen}. \\

Let $A$ be a C*-algebra (not necessarily unital). Define an associative composition $\cdot$ on $A$ by
\begin{equation}\label{grp_oper}
a\cdot b=a+b-ab
\end{equation}
An element $u\in A$ is said to be a quasi-unitary if
\[
u\cdot u^{\ast} = u^{\ast}\cdot u = 0.
\]
We write $\widehat{\mathcal{U}}(A)$ for the set of all quasi-unitary elements in $A$. For elements $u,v\in \U(A)$, we write $u\sim v$ if there is a continuous function $f:[0,1]\to \U(A)$ such that $f(0) = u$ and $f(1) = v$. We write $\U_0(A)$ for the set of $u\in \U(A)$ such that $u\sim 0$. Note that $\U_0(A)$ is a closed, normal subgroup of $\U(A)$. We now define the two functors we are interested in.
  
\begin{defn}
Let $A$ be a $C^*$-algebra, and $k\geq 0$ and $m\geq 1$ be integers. Define
\[
G_k(A) :=\pi_k(\U(A)), \text{ and } F_m(A) := \pi_m(\U_0(A))\otimes \Q \cong G_m(A)\otimes \Q.
\]
\end{defn}

Recall \cite{schochet} that a homology theory on the category of $C^{\ast}$-algebras is a sequence $\{h_n\}$ of covariant, homotopy invariant functors from the category of $C^*$-algebras to the category of abelian groups such that, if $0 \to J\xrightarrow{\iota} B\xrightarrow{p} A\to 0$ is a short exact sequence of $C^*$-algebras, then for each $n \in \N$, there exists a connecting map $\partial : h_n(A)\to h_{n-1}(J)$, making the following sequence exact
\[
\ldots \xrightarrow{\partial} h_n(J)\xrightarrow{h_n(\iota)} h_n(B) \xrightarrow{h_n(p)} h_n(A)\xrightarrow{\partial} h_{n-1}(J)\to \ldots
\]
and furthermore, $\partial$ is natural with respect to morphisms of short exact sequences. Furthermore, we say that a homology theory $\{h_n\}$ is continuous if, whenever $A = \lim A_i$ is an inductive limit in the category of $C^*$-algebras, then $h_n(A) = \lim h_n(A_i)$ in the category of abelian groups. The next proposition is a consequence of \cite[Proposition 2.1]{thomsen} and \cite[Theorem 4.4]{handelman}.

\begin{prop}\label{prop: continuous_homology}
For each $m\geq 1$, $G_m$ and $F_m$ are continuous homology theories.
\end{prop}

The notion of $K$-stability given below is due to Thomsen \cite[Definition 3.1]{thomsen}, and that of rational $K$-stability has been studied by Farjoun and Schochet \cite[Definition 1.2]{farjoun}, where it was termed rational Bott-stability.

\begin{defn}\label{defn:k_stable}
Let $A$ be a $C^*$-algebra and $j\geq 2$. Define $\iota_j: M_{j-1}(A)\to M_j(A)$ to be the natural inclusion map
\[
a\mapsto\begin{pmatrix}
a&0\\
0&0
\end{pmatrix}.
\]
$A$ is said to be \emph{$K$-stable} if $G_k(\iota_j): G_k(M_{j-1}(A))\to G_k(M_j(A))$ is an isomorphism for all $k\geq 0$ and all $j\geq 2$. $A$ is said to be \emph{rationally $K$-stable} if $F_m(\iota_j):F_m(M_{j-1}(A))\to F_m(M_j(A))$ is an isomorphism for all $m\geq 1$ and all $j\geq 2$.
\end{defn}

Note that, for a $K$-stable $C^*$-algebra, $G_k(A) \cong K_{k+1}(A)$ and for a rationally $K$-stable $C^*$-algebra, $F_m(A) \cong K_{m+1}(A)\otimes \Q$. A variety of interesting C*-algebras are known to be $K$-stable (see \cite[Remark 1.5]{vaidyanathan}). Clearly, $K$-stability implies rational $K$-stability. By \cite[Theorem B]{seth}, the converse is true for AF-algebras. However, as \cref{ex: rat_k_stable_not_k_stable} shows, the converse is not true in general.

\subsection{$C(X)$-algebras}

Let $A$ be a $C^*$-algebra, and $X$ a compact Hausdorff space. We say that $A$ is a $C(X)$-algebra \cite[Definition 1.5]{Kasparov} if there is a unital $\ast$-homomorphism $\theta : C(X)\to \mathcal{Z}(M(A))$, where $\mathcal{Z}(M(A))$ denotes the center of the multiplier algebra of $A$. For simplicity of notation, if $f\in C(X)$ and $a\in A$, we write $fa := \theta(f)(a)$.\\

If $Y\subset X$ is closed, the set $C_0(X,Y)$ of functions in $C(X)$ that vanish on $Y$ is a closed ideal of $C(X)$. Hence, $C_0(X,Y)A$ is a closed, two-sided ideal of $A$. The quotient of $A$ by this ideal is denoted by $A(Y)$, and we write $\pi_Y : A\to A(Y)$ for the quotient map (also referred to as the restriction map). If $Z\subset Y$ is a closed subset of $Y$, we write $\pi^Y_Z : A(Y)\to A(Z)$ for the natural restriction map, so that $\pi_Z = \pi^Y_Z\circ \pi_Y$. If $Y = \{x\}$ is a singleton, we write $A(x)$ for $A(\{x\})$ and $\pi_x$ for $\pi_{\{x\}}$. The algebra $A(x)$ is called the fiber of $A$ at $x$. For $a\in A$, write $a(x)$ for $\pi_x(a)$. For each $a\in A$, we have a map
\[
\Gamma_a : X\to \R \text{ given by } x \mapsto \|a(x)\|.
\]
This map is, in general, upper semi-continuous \cite[Lemma 2.3]{kw_fields}. We say that $A$ is a \emph{continuous} $C(X)$-algebra if $\Gamma_a$ is continuous for each $a\in A$.\\

If $A$ is a $C(X)$-algebra, we will often have reason to consider other $C(X)$-algebras obtained from $A$. At that time, the following result of Kirchberg and Wasserman will be useful.
\begin{theorem}\cite[Remark 2.6]{kw_fields}\label{thm:kw_tensor}
Let $X$ be a compact Hausdorff space, and let $A$ be a continuous $C(X)$-algebra. If $B$ is a nuclear $C^*$-algebra, then $A\otimes B$ is a continuous $C(X)$-algebra whose fiber at a point $x\in X$ is $A(x)\otimes B$.
\end{theorem}

In particular, if $A$ is a continuous $C(X)$-algebra, then so is $M_2(A)$. If $Y\subset X$ is a closed set, we will denote the restriction map by $\eta_Y : M_2(A)\to M_2(A(Y))$, and we write $\iota_Y : A(Y)\to M_2(A(Y))$ for the natural inclusion map. If $Y = X$, we simply write $\iota$ (or $\iota^A$) for $\iota_X$. Note that $\eta_Y \circ \iota = \iota_Y \circ \pi_Y$. Once again, if $Y = \{x\}$, we simply write $\iota_x$ for $\iota_{\{x\}}$. \\

Finally, the notion of a pullback is important for our investigation: Let $B,C,$ and $D$ be $C^*$-algebras, and $\delta : B\to D$ and $\gamma : C\to D$ be $\ast$-homomorphisms. We define the pullback of this system to be
\[
A = B\oplus_D C := \{(b,c) \in B\oplus C : \delta(b) = \gamma(c)\}.
\]
This is described by a diagram
\begin{equation}\label{eqn:pullback_diag}
\xymatrix{
A\ar[r]^{\phi}\ar[d]_{\psi} & B\ar[d]^{\delta} \\
C\ar[r]^{\gamma} & D
}
\end{equation}
where $\phi(b,c) = b$ and $\psi(b,c) = c$. The next lemma allows us to inductively put together a $C(X)$-algebra from its natural quotients. 
\begin{lem}\cite[Lemma 2.4]{mdd_finite}\label{lem:cx_algebra_pullback}
Let $X$ be a compact Hausdorff space and $Y$ and $Z$ be two closed subsets of $X$ such that $X = Y\cup Z$. If $A$ is a $C(X)$-algebra, then $A$ is isomorphic to the pullback
\[
\xymatrixcolsep{1.5cm}\xymatrix{
A\ar[r]^{\pi_Y}\ar[d]_{\pi_Z} & A(Y)\ar[d]^{\pi^Y_{Y\cap Z}} \\
A(Z)\ar[r]^{\pi^Z_{Y\cap Z}} & A(Y\cap Z).
}
\]
\end{lem}

\subsection{Rational Homotopy Theory}

We now discuss some basic facts about the rationalization of groups and spaces as developed in \cite{hilton}. \\

A connected CW-complex $Y$ is said to be nilpotent if $\pi_1(Y)$ is a nilpotent group and $\pi_1(Y)$ acts nilpotently on $\pi_j(Y)$ for all $j\geq 2$. A nilpotent space $Y$ is a rational space if, for each $j\geq 1$, the homotopy group $\pi_j(Y)$ is a $\Q$-vector space. A continuous map $r:Y\to Z$ is said to be a rationalization of $Y$ if $Z$ is a rational space and $r_{\ast}\otimes \text{id} : \pi_{\ast}(Y)\otimes \Q \to \pi_{\ast}(Z)\otimes \Q\cong \pi_{\ast}(Z)$ is an isomorphism. The next theorem (see \cite[Theorem II.3A]{hilton}) is fundamental to the theory.

\begin{theorem}[Hilton, Mislin, and Roitberg]\label{thm: hilton}
Every nilpotent CW complex $Y$ has a rationalization $r:Y\rightarrow Y_{\Q}$, where $Y_{\Q}$ is a CW complex. The space $Y_{\Q}$ is uniquely determined up to homotopy equivalence.
\end{theorem}

We now specialize to the situation of our interest. Recall that an $H$-space is a pointed space $(Y,e)$ endowed with a `multiplication' map $\mu:Y\times Y\to Y$ such that $e$ is a homotopy unit, that is, the maps $\lambda, \rho :Y\to Y$ given by $\lambda(y) := \mu(e,y)$ and $\rho(y) := \mu(y,e)$ are both homotopic to $\text{id}_Y$. We denote this by the triple $(Y,e,\mu)$. We say that $(Y,e,\mu)$ is homotopy-associative if the maps $\mu\circ (\mu\times \text{id}_Y)$ and $\mu\circ (\text{id}_Y\times \mu) : Y\times Y \times Y \to Y$ are homotopic. In what follows, we will implicitly assume that the $H$-spaces under consideration are all homotopy-associative.\\

Now suppose $(Y,e,\mu)$ is an $H$-space, where the space $Y$ is a connected CW-complex. Since $Y$ is nilpotent, it has a rationalization $r:Y\to Y_{\Q}$ by \cref{thm: hilton}. Now, by \cite[Theorem 6.2.3]{may_ponto}, $r\times r : Y\times Y\to Y_{\Q}\times Y_{\Q}$ is a rationalization. By the universal property of the rationalization, there is a unique map $\rho : Y_{\Q}\times Y_{\Q}\to Y_{\Q}$ such that the following diagram commutes upto homotopy.
\begin{equation}\label{diag: rationalization_h_space}
\xymatrix{
Y\times Y\ar[r]^{\mu}\ar[d]_{r\times r} & Y\ar[d]^{r} \\
Y_{\Q}\times Y_{\Q}\ar[r]^{\rho} & Y_{\Q}.
}
\end{equation}
By the mapping cylinder construction, we may assume that $r$ is a cofibration. Then, $r\times r$ is also a cofibration as it is the composition of two cofibrations $Y\times Y\to Y\times Y_{\Q}\to Y_{\Q}\times Y_{\Q}$. Hence, by \cite[Problem 5.3]{strom}, we may assume that the above diagram commutes strictly. If we set $e_{\Q} := r(e)$, then it follows from \cite[Proposition 6.6.2]{may_ponto} that the triple $(Y_{\Q},e_{\Q},\rho)$ is an $H$-space. Furthermore, by universality, we may also ensure that the triple $(Y,e_{\Q}, \rho)$ is homotopy-associative. We summarize this result below.

\begin{prop}\label{prop: h_space_rationalization}
If $(Y,e,\mu)$ is a homotopy-associative $H$-space, where $Y$ is a connected CW-complex, then there is a homotopy-associative $H$-space $(Y_{\Q},e_{\Q},\rho)$ and a map $r:Y\to Y_{\Q}$ such that $r$ is a rationalization, and the diagram \cref{diag: rationalization_h_space} commutes strictly.
\end{prop}

If $A$ is a C*-algebra, then $\U(A)$ has the homotopy type of a CW-complex \cite[Corollary 1.6]{thomsen}. Therefore, $\U_0(A)$ may be regarded as a connected CW-complex. Since $\U_0(A)$ is a topological group (and hence a connected $H$-space), it has as rationalization $r:\U_0(A)\to \U_0(A)_{\Q}$. By \cref{prop: h_space_rationalization}, $\U_0(A)_{\Q}$ has the structure of an $H$-space, which we write as $(\U_0(A)_{\Q},e_{\Q},\rho)$, where $e_{\Q} = r(0)$. Finally, observe that the commutativity of \cref{diag: rationalization_h_space} implies that $\rho(e_{\Q},e_{\Q}) = e_{\Q}$.

\subsection{Notational Conventions}
If $A$ and $B$ are two $C^*$-algebras, the symbol $A\otimes B$ will always denote the minimal tensor product. If $B = C_0(X)$ is commutative, we identify $C_0(X)\otimes A$ with $C_0(X,A)$, the space of continuous $A$-valued functions on $X$ that vanish at infinity. \\

Suppose $f$ and $g$ are two continuous paths in a topological space $Y$. If $f(1) = g(0)$, we write $f\bullet g$ for the concatenation of the two paths. If $f$ and $g$ agree at end-points, we write $f\sim_h g$ if there is a path homotopy between them. Furthermore, we write $\overline{f}$ for the path $\overline{f}(t) := f(1-t)$ and the constant path at a point $\ast$ as $e_{\ast}$.\\

If $X$ and $Y$ are two pointed spaces, we write $C_{\ast}(X,Y)$ for the space of base point preserving continuous functions from $X$ to $Y$. Note that if $A$ is a C*-algebra, and $Y$ is either $A$ or $\U_0(A)$, then we always take $0$ to be the base point. In that case, $C_{\ast}(X,A)$ is a C*-algebra, and, for any path-connected space $X$, there is a natural isomorphism
\[
\U(C_{\ast}(X,A)) \cong C_{\ast}(X,\U_0(A)).
\]
Henceforth, we will identify these two spaces without further comment. \\

If $(Y,e,\mu)$ is an $H$-space and $a\in Y$, we may define non-negative powers of $a$ inductively by $\mu_0(a) := e$ and $\mu_n(a) := \mu(\mu_{n-1}(a),a)$. Similarly, if $f:X\to Y$ is any function, we define non-negative powers of $f$ point-wise, that is, $\mu_n(f)(x) := \mu_n(f(x))$ for all $n\geq 0$. Note that, if $f\in C_{\ast}(S^j,Y)$, then $[\mu_n(f)] = n[f]$ in $\pi_j(Y)$ by \cite[Theorem 4.7]{whitehead}. Throughout the rest of the paper, for any $C^*$-algebra $B$, we write $\mu^B$ for the multiplication in $\U_0(B)$ given by \cref{grp_oper}, and $\rho^B$ for the multiplication in $\U_0(B)_{\Q}$ given by \cref{prop: h_space_rationalization}.

\section{Main Results}\label{sec: main}

The goal of this section is to provide a proof for \cref{main_thm}. To put things in perspective, we begin by constructing an example of a C*-algebra that is rationally $K$-stable, but not $K$-stable.

\begin{ex}\label{ex: rat_k_stable_not_k_stable}
Let $X$ be a connected, finite CW-complex such that $H^i(X;\Z)$ is a finite group for all $i\geq 1$ (for instance, we may take $X$ to be the real projective space $\R\mathbb{P}^2$), and set
\[
A := C_{\ast}(X,\C).
\]
Note that, for all $n,m\geq 1$,
\[
F_n(M_m(A))=\pi_n(C_{\ast}(X;\mathcal{U}_m))\otimes\Q = \bigoplus_{l\geq n} \tilde{H}^{l-n}(X;\pi_l(\mathcal{U}_m)\otimes\Q)=0
\]
by \cite[Theorem 4.20]{lupton}. Hence, $A$ is rationally $K$-stable. \\

Now suppose that $A$ is $K$-stable. We fix a path connected $H$-space $Y$, and consider the following fibration sequence (see the proof of \cite[Proposition 4.9]{lupton})
\[
C_{\ast}(X,Y)\to C(X,Y)\to Y.
\]
This fibration has a section, hence the long exact homotopy sequence breaks into split short exact sequences 
\begin{equation}\label{eqn: ses_homotopy_group_h_space}
0 \to \pi_n(C_{\ast}(X,Y))\to \pi_n(C(X,Y))\to \pi_n(Y)\to 0
\end{equation}
for all $n\in \N$. By a result of Thom \cite[Theorem 2]{thom}, if $Y = S^1 = K(\Z,1)$, then $\pi_n(C(X,S^1)) \cong H^{1-n}(X;\Z)$. It follows that
\[
G_n(A) = \pi_n(\U(C_{\ast}(X,\C))) \cong \pi_n(C_{\ast}(X,S^1)) = 0
\]
for all $n\geq 1$. If $A$ were $K$-stable, it would follow that
\[
\pi_n(C_{\ast}(X,\mathcal{U}_m)) \cong G_n(M_m(A)) \cong G_n(A) = 0
\]
for all $n,m\geq 1$. Hence, we conclude that $\pi_n(C_{\ast}(X,\U(\mathcal{K}))) \cong G_n(A\otimes \mathcal{K}) =0$ for all $n\geq 1$. Taking $Y = \U(K)$ in \cref{eqn: ses_homotopy_group_h_space}, we conclude that
\begin{equation}\label{eqn: n_group}
\pi_n(C(X,\U(\mathcal{K})))\cong \pi_n(\U(\mathcal{K}))=\begin{cases}
\Z &:  n \text{ odd} \\
0 &: n \text{ even}
\end{cases}.
\end{equation}
Thus, in order to show that $A$ is not $K$-stable, it suffices to show that \cref{eqn: n_group} cannot hold. To do this, we consider the work of Federer \cite{federer}, who constructed a spectral sequence converging to these homotopy groups (note that $X$ is a finite CW-complex, and $\U(K)$ is a simple space, so the results of \cite{federer} do apply). The first page of this spectral sequence, which converges to $\pi_p(C(X,\U(K)))$, is of the form
\[
C^{(1)}_{p,q}\cong H^q(X;\pi_{p+q}(\U(K)))
\]
with differential $d:C^{(1)}_{p,q}\to C^{(1)}_{p-1,q+2}$. Therefore, for $C^{(1)}_{p,q}$ to be non-zero, $p+q$ must be odd. But in that case, $C^{(1)}_{p-1,q+2}$ is zero. Hence, the spectral sequence collapses at the very first page, so $C^{(1)}_{p,q} = C^{(\infty)}_{p,q}$. Therefore, we conclude that
\[
\pi_n(C(X,\U(K))) = \bigoplus_{q\geq 0} H^q(X;\pi_{n+q}(\U(K)))
\]
for all $n\geq 1$. This is a finite sum of finite groups (by our choice of $X$), contradicting \cref{eqn: n_group}. Thus, $A$ is not $K$-stable.
\end{ex}

We now turn to the proof of \cref{main_thm}, and begin with some lemmas that will be useful to us. The first lemma, which we will use repeatedly throughout the paper, follows from \cite[Theorem 1.9]{thomsen} and \cite[Theorem 4.8]{dold}.

\begin{lem}\label{thm:fibration}
Let $\varphi : A\to B$ be a surjective $\ast$-homomorphism between two C*-algebras, then the induced maps $\varphi : \U(A)\to \varphi(\U(A))$ and $\varphi : \U_0(A)\to \U_0(B)$ are both Serre fibrations.
\end{lem}

\begin{lem}\cite[Lemma 2.2]{vaidyanathan}\label{lem:quasi_homotopy}
Let $a,b\in \U(A)$ such that $\|a-b\| < 2$, then $a\sim b$ in $\U(A)$.
\end{lem}

Note that, for any element $a$ in a C*-algebra $A$ (not necessarily a quasi-unitary), we write $\mu^A_N(a)$ for $a\cdot a\cdot \ldots a$ ($N$ times). The next lemma is a variation of \cite[Lemma 2.3]{vaidyanathan} that we need for our purposes.

\begin{lem}\label{lem:quasi_unitary_close}
For any $\epsilon > 0$ and any $N\in \N$, there exists $\delta > 0$ satisfying the following condition: For any C*-algebra $A$, and any element $a \in A$ such that $\|a\| \leq 2, \|a\cdot a^{\ast}\| < \delta$, and $\|a^{\ast}\cdot a\| < \delta$, there exists a quasi-unitary $u\in \U(A)$ such that
\[
\|\mu^A_N(u) - \mu^A_N(a)\| < \epsilon.
\]
\end{lem}
\begin{proof}
Note that the function $d\mapsto \mu^A_N(d)$ is a polynomial in $d$ (that is independent of $A$). Thus, for any $\epsilon > 0$, there exists $\eta > 0$ satisfying the following condition: For any C*-algebra $A$ and any $c,d\in A$ with $\|c\|, \|d\|\leq 2$ such that $\|c-d\| < \eta$, we have $\|\mu^A_N(c) - \mu^A_N(d)\| < \epsilon$. \\

We choose $\delta > 0$ satisfying the conditions of \cite[Lemma 2.3]{vaidyanathan} with $\epsilon = \eta$, then there exists $u \in \U(A)$ such that $\|u-a\| < \eta$, so that $\|\mu^A_N(u) - \mu^A_N(a)\| < \epsilon$.
\end{proof}

Our proof of \cref{main_thm} is by induction on the covering dimension of the underlying space. The next theorem is the base case, and it holds even if the space is not metrizable. In what follows, we will repeatedly use the fact that, for any abelian group $A$, any element in $A\otimes\Q$ can be represented as an elementary tensor of the form $u\otimes 1/m$ for some $u\in A$ and $m\in \Z$. 

\begin{theorem}\label{thm:dim_zero_case_}
Let $X$ be a compact Hausdorff space of zero covering dimension, and let $A$ be a continuous $C(X)$-algebra. If each fiber of $A$ is rationally $K$-stable, then so is $A$.
\end{theorem}
\begin{proof}
We show that the map
\[
\iota_{\ast}\otimes \text{id} : \pi_j(\U_0(A))\otimes\Q \to \pi_j(\U_0(M_n(A)))\otimes\Q
\]
is an isomorphism for each $n\geq 2$ and $j\geq 1$. For simplicity of notation, we fix $n=2$.\\

We first consider injectivity. Suppose $[f]\otimes q\in \pi_j(\U_0(A))\otimes\Q$ such that $[\iota\circ f]\otimes q =0$ in $\pi_j(\U_0(M_2(A))\otimes\Q$. Then by elementary group theory $[\iota\circ f]$ has finite order in $\pi_j(\U_0(M_2(A)))$. Thus for $x\in X$, $[\iota_x\circ\pi_x\circ f]$ has finite order in $\pi_j(\U_0(M_2(A(x)))$.  Since $A(x)$ is rationally $K$-stable, $[\pi_x\circ f]$ has finite order in $\pi_j(\U_0(A(x)))$. Hence, there exists $N_x\in\N$ and a path $F:[0,1]\to C_*(S^j,\U_0(A(x)))$ such that
\[
F(0)=0, \text{ and } F(1)=\mu^{A(x)}_{N_x}(\pi_x\circ f).
\]
By \cite[Lemma 2.4]{vaidyanathan}, there is a closed neighbourhood $Y_x$ of $x$ such that $\mu^{A(Y_x)}_{N_x}(\pi_{Y_x}\circ f) \sim 0$ in $C_{\ast}(S^j,\U_0(A(Y_x)))$. Since $X$ is zero dimensional, we may assume that the sets $\{Y_x : x\in X\}$ are cl-open and disjoint. Since $X$ is compact, we may obtain a finite sub-cover $\{Y_{x_1}, Y_{x_2}, \ldots, Y_{x_n}\}$. By \cref{lem:cx_algebra_pullback}, 
\[
A \cong A(Y_{x_1})\oplus A(Y_{x_2})\oplus \ldots \oplus A(Y_{x_n})
\]
via the map $b\mapsto (\pi_{Y_{x_1}}(b), \pi_{Y_{x_2}}(b), \ldots, \pi_{Y_{x_n}}(b))$. If $N:=\text{lcm}_{1\leq i\leq n}(N_{x_i})$, then $\mu^{A(Y_{x_i})}_N(\pi_{Y_{x_i}}\circ f)\sim 0$ in $C_*(S^j,\U_0(A(Y_{x_i})))$, for each $1\leq i\leq n$. Thus, $\mu_N(f) \sim 0$ in $\pi_j(\U_0(A))$. Hence, $[f]$ has finite order in $\pi_j(\U_0(A))$, so $[f]\otimes q=0$ in $\pi_j(\U_0(A))\otimes\Q$. Thus, $\iota_{\ast}\otimes\text{id}$ is injective. \\ 

For surjectivity, choose $[u]\in \pi_j(\U_0(M_2(A)))$ and $m\in \Z$ non-zero. We wish to construct an element $[\omega] \in \pi_j(\U_0(A))$ and $q\in \Q$ such that
\[
\iota_{\ast}\otimes\text{id}([\omega]\otimes q)=[u]\otimes\frac{1}{m}.
\]
To this end, fix $x\in X$. Since $A(x)$ is rationally $K$-stable, there exists $[f_x] \in \pi_j(\U_0(A(x)))$ and $q_x \in \Q$ such that
\[
(\iota_x)_{\ast}\otimes \text{id}([f_x]\otimes q_x) = [\eta_x\circ u]\otimes\frac{1}{m}.
\]
Replacing $f_x$ by a multiple of itself if need be, we obtain integers $L_x, N_x\in \N$ such that
\[
N_x[\iota_x\circ f_x] = L_x[\eta_x\circ u]
\]
in $\pi_j(\U_0(M_2(A(x))))$. Hence, there is a path $g_x : [0,1]\to C_{\ast}(S^j,\U_0(M_2(A(x))))$ such that $g_x(0)  = \mu^{M_2(A(x))}_{L_x}(\eta_x\circ u)$ and $g_x(1) = \mu^{M_2(A(x))}_{N_x}(\iota_x\circ f_x)$. Choose $e_x \in C_{\ast}(S^j,A)$ such that $\pi_x\circ e_x = f_x$. Note that $e_x$ may not be a quasi-unitary, but we may ensure that $\|e_x\| = \|f_x\| \leq 2$. Since the map
\[
\eta_x : C_{\ast}(S^j,\U_0(M_2(A))) \to \eta_x(C_{\ast}(S^j,\U_0(M_2(A))))
\]
is a fibration, $g_x$ lifts to a path $G_x : [0,1]\to C_{\ast}(S^j,\U_0(M_2(A)))$ such that $G_x(0) = \mu^{M_2(A)}_{L_x}(u)$. Let $b_x := G_x(1)$, and  so that $\eta_x \circ b_x = \mu^{M_2(A(x))}_{N_x}(\iota_x\circ\pi_x\circ e_x)$. 
Choose $\delta > 0$ so that conclusion of \cref{lem:quasi_unitary_close} holds for $\epsilon = 1$ and $N=N_x$. Since $A$ is a continuous $C(X)$-algebra, there is a closed neighbourhood $Y_x$ of $x$ such that
\[
\|\pi_{Y_x}\circ(e_x^{\ast}\cdot e_x)\| < \delta, \|\pi_{Y_x}\circ(e_x\cdot e_x^{\ast})\| < \delta
\]
and $\|\eta_{Y_x}\circ b_x -\mu^{M_2(A(Y_x))}_{N_x}(\eta_{Y_x}\circ\iota\circ e_x)\| < 1$. By \cref{lem:quasi_unitary_close}, there is a quasi-unitary $d_x \in C_{\ast}(S^j,\U_0(A(Y_x)))$ such that $\|\mu^{A(Y_x)}_{N_x}(d_x) - \mu^{A(Y_x)}_{N_x}(\pi_{Y_x}\circ e_x)\| < 1$, so that
\[
\|\mu^{M_2(A(Y_x))}_{N_x}(\iota_{Y_x}\circ d_x) - \eta_{Y_x}\circ b_x\| < 2.
\]
By \cref{lem:quasi_homotopy}, $\mu^{M_2(A(Y_x))}_{N_x}(\iota_{Y_x}\circ d_x) \sim \eta_{Y_x}\circ b_x$ in $C_{\ast}(S^j,\U_0(M_2(A(Y_x)))$. Hence, $\iota_{Y_x}\circ \mu^{A(Y_x)}_{N_x}(d_x) \sim \mu^{M_2(A(Y_x))}_{L_x}(\eta_{Y_x}\circ u)$. As before, since $X$ is compact and zero-dimensional, we may choose a finite refinement of $\{Y_x : x\in X\}$ consisting of disjoint cl-open sets, which we denote by $\{Y_{x_1}, Y_{x_2},\ldots, Y_{x_n}\}$. Then, by \cref{lem:cx_algebra_pullback},
\[
A \cong A(Y_{x_1})\oplus A(Y_{x_2})\oplus \ldots \oplus A(Y_{x_n})
\]
via the map $a\mapsto (\pi_{Y_{x_1}}(a), \pi_{Y_{x_2}}(a), \ldots, \pi_{Y_{x_n}}(a))$. Similarly,
\[
M_2(A) \cong M_2(A(Y_{x_1}))\oplus M_2(A(Y_{x_2}))\oplus \ldots \oplus M_2(A(Y_{x_n}))
\]
via the map $b \mapsto (\eta_{Y_{x_1}}(b), \eta_{Y_{x_2}}(b),\ldots, \eta_{Y_{x_n}}(b))$. Define $L:=\text{lcm}_{1\leq i\leq n}(L_{x_i})$, so that
\[
\iota_{Y_{x_i}}\circ c_{x_i}\sim \mu^{M_2(A(Y_{x_i}))}_L(\eta_{Y_{x_i}}\circ u)
\]
in $C_{\ast}(S^j,\U_0(M_2(A(Y_{x_i})))$, where $c_{x_i}\in C_{\ast}(S^j,\U_0(A(Y_{x_i})))$ is an appropriate power of $d_{x_i}$. Choose $\omega\in C_{\ast}(S^j,\U_0(A))$ such that $\pi_{Y_{x_i}}\circ\omega = c_{x_i}$ for all $1\leq i\leq n$. Furthermore, for each $1\leq i\leq n$,
\[
\eta_{Y_{x_i}}\circ\iota\circ\omega = \iota_{Y_{x_i}}\circ c_{x_i} \sim \mu^{M_2(A(Y_{x_i}))}_L(\eta_{Y_{x_i}}\circ u)
\]
in $C_{\ast}(S^j,\U_0(M_2(A(Y_{x_i})))$, so that $\iota\circ\omega \sim \mu^{M_2(A)}_L(u)$ in $C_{\ast}(S^j,\U_0(M_2(A)))$. Thus,
\[
\iota_{\ast}\otimes \text{id}\bigg([\omega]\otimes\frac{1}{Lm}\bigg)=[u]\otimes\frac{1}{m}.
\]
This proves the surjectivity of $\iota_{\ast}\otimes \text{id}$.
\end{proof}

The next few lemmas allow us to extend this argument to higher dimensional spaces.
%

\begin{lem}\label{lem:rational_homotopy}
Let $(Y,e, \mu)$ be an H-space, where $Y$ is a connected CW-complex. Let $r:(Y,e,\mu)\to (Y_{\Q},e_\Q,\rho)$ be the rationalization map from \cref{prop: h_space_rationalization}, and let $j\geq 1$ be a fixed integer. 
\begin{enumerate}
\item Let $[f]\in \pi_j(Y)$ and $n\in \N$, and suppose there is a path $H:[0,1]\to C_{\ast}(S^j,Y)$ such that $H(0) = e$ and $H(1) = \mu_n(f)$. Then, there exists a path $G:[0,1]\to C_{\ast}(S^j,Y_{\Q})$ with $G(0) = e_{\Q}$ and $G(1) = r\circ f$, and such that
\[
r\circ H\sim_h \rho_n(G)
\]
in $C_{\ast}(S^j,Y_{\Q})$. 
\item Let $[f] \in \pi_j(Y)$, and suppose there is a path $G':[0,1]\to C_{\ast}(S^j,Y_{\Q})$ such that $G'(0) = e_{\Q}$ and $G'(1) = r\circ f$. Then, there exists a natural number $N\in\N$ and a path $H':[0,1]\to C_{\ast}(S^j,Y)$ with $H'(0) = e$ and $H'(1) = \mu_N(f)$, such that
\[
r\circ H'\sim_h \rho_N(G')
\]
in $C_{\ast}(S^j,Y_{\Q})$.
\end{enumerate}
\end{lem}
\begin{proof}
~\begin{enumerate}
\item Since $Y_{\Q}$ is a rational space, $[r\circ f] = 0$ in $\pi_j(Y_{\Q})$. Let $L:[0,1]\to C_{\ast}(S^j,Y_{\Q})$ be such that $L(0) = e_{\Q}$ and $L(1) = r\circ f$. Thus, $\rho_n(L) : [0,1]\to C_{\ast}(S^j,Y_{\Q})$ is a path that satisfies $\rho_n(L)(0) = e_{\Q}$ and $\rho_n(L)(1) = \rho_n(r\circ f)$. Note that $\pi_1(C_{\ast}(S^j,Y_{\Q}))$ is itself a $\Q$-vector space \cite[Theorem II.3.11]{hilton} and $(r\circ H)\bullet \overline{\rho_n(L)}$ is a loop in $C_{\ast}(S^j,Y_{\Q})$. Thus, there exists $[T]\in \pi_1(C_{\ast}(S^j,Y_{\Q}))$ such that
\[
[(r\circ H)\bullet \overline{\rho_n(L)}] = n[T] = [\rho_n(T)].
\]
Hence, $G := T\bullet L$ is the required homotopy (since the operation $\rho_n$ respects concatenation).
\item Since $[r\circ f]=0$, in $\pi_j(Y_{\Q})$, under $r_{\ast}\otimes \Q:\pi_j(Y)\otimes\Q\to\pi_j(Y_{\Q})$ we get that $[r\circ f]\cong [f]\otimes 1=0$ in $\pi_j(Y)\otimes\Q$. Hence by elementary group theory, this implies that $[f]$ has finite order in $\pi_j(Y)$. Thus, there exists $n\in\N$ such that $n[f]=0$ in $\pi_j(Y)$ say by homotopy $K:[0,1]\to C_{\ast}(S^j,Y)$ such that
\[
K(0)=e, \quad K(1)= \mu_n(f).
\]
Now, by a similar argument to that of part (1), $(r\circ K)\bullet \overline{\rho_n(G')}$ is a loop in $C_{\ast}(S^j,Y_{\Q})$, which is a rational space. Hence, there exists $[T] \in \pi_1(C_{\ast}(S^j,Y_{\Q}))$ satisfying
\[
n[T]=[\rho_n(T)] = [r\circ K\bullet \overline{\rho_n(G')}].
\]
Now $n[T]\in\pi_1(C_{\ast}(S^j,Y_{\Q}))\cong\pi_1(C_{\ast}(S^j,Y))\otimes\Q$, so there exists $[h] \in \pi_1(C_{\ast}(S^j,Y))$ and $m \in \Z$ such that
\[
n[T] = \frac{[r\circ h]}{m}.
\]
Thus, by the fact that $\rho$ is homotopy-associative,
\[
m[r\circ K\bullet\overline{\rho_n(G')}]=[\rho_m(r\circ K)\bullet\overline{\rho_{mn}(G')}]=mn[T]=[r\circ h].
\]
Thus, if $H':=\overline{h}\bullet \mu_m(K)$ and $N:=mn$, then $H'(0)=e, H'(1)=\mu_N(f)$, and $r\circ H'$ is path homotopic to $\rho_N(G')$.\qedhere
\end{enumerate}
\end{proof}

The next result will be useful to us in the following context: Suppose $B,C,$ and $D$ be $C^*$-algebras, and $\delta : B\to D$ and $\gamma : C\to D$ are $\ast$-homomorphisms. Let $A = B\oplus_D C$ be the pullback as in \cref{eqn:pullback_diag}. Then, $\U(A)$ may be described as a pullback (in the category of pointed topological spaces) by the induced diagram
\[
\xymatrix{
\U(A)\ar[r]^{\phi}\ar[d]_{\psi} & \U(B)\ar[d]_{\delta} \\
\U(C)\ar[r]_{\gamma} & \U(D)
}
\]
In other words, a pair $(b,c)\in A$ is in $\U(A)$ if and only if $b\in \U(B)$ and $c\in \U(C)$. We now introduce some notation we will use in the future: Given a path $G:[0,1]\to Y$ in a topological space $Y$, $\widetilde{G}$ is a path given by 
\begin{equation}\label{eqn: tilde_path}
  \widetilde{G}(s) =
  \begin{cases}
    e_{G(0)}(3s) &: \text{for } 0 \leq x \leq \frac{1}{3}\\
    G(3s-1) &: \text{for } \frac{1}{3} \leq x \leq \frac{2}{3} \\
    e_{G(1)}(3s-2) &: \text{for } \frac{2}{3} \leq x \leq 1
  \end{cases}
\end{equation}

\begin{lem}\label{hom_pullback}
Consider a pullback diagram of pointed topological spaces given by
\[
\xymatrix{P\ar[r]^{\phi_1}\ar[d]_{\phi_2}& X\ar[d]^{\pi_1}\\
Y\ar[r]_{\pi_2}& Z}
\]such that one of the maps $\pi_1$ or $\pi_2$ is a Serre fibration. Let $p=(x,y)$, $p'=(x',y')$ be in $P$, such that there exists paths 
\[
G_1:[0,1]\to X,\quad G_2:[0,1]\to Y
\]with the property that $G_1(0)=x$, $G_1(1)=x'$, $G_2(0)=y$, $G_2(1)=y'$ and $\pi_1\circ G_1\sim_h \pi_2\circ G_2$ in $Z$. Then, there is a path $H:[0,1]\to P$ such that $H(0) = p$ and $H(1) = p'$. 
\end{lem}
\begin{proof}
Assume without loss of generality that $\pi_1$ is a Serre fibration. Then since, $\pi_1\circ G_1\sim_h \pi_2\circ G_2$, there is a homotopy $F:[0,1]\times[0,1]\to D$ such that 
\begin{equation*}
\begin{split}
F(s,0)=\pi_1\circ G_1(s), &\quad F(s,1)=\pi_2\circ G_2(s)\\
F(0,t)=\pi_1(x)=\pi_2(y), &\quad F(1,t)=\pi_1(x')=\pi_2(y').
\end{split}
\end{equation*}
Then, $F$ lifts to a homotopy $F':[0,1]\times [0,1]\to X$, such that 
\[
F'(s,0)=G_1,\quad \pi_1\circ F'=F,\quad \pi_1\circ F'(t,1)=\pi_2\circ G_2(t).
\]
So if we define
\[
  {G_X}(s) =
  \begin{cases}
    F'(0,3s), & \text{for } 0 \leq s \leq \frac{1}{3}\\
    F'(3s-1,1), & \text{for } \frac{1}{3} \leq s \leq \frac{2}{3} \\
    F'(1,3-3s), & \text{for } \frac{2}{3} \leq s \leq 1
  \end{cases}
\]
then $\pi_1\circ G_X=\pi_2\circ \widetilde{G_2}$. Therefore, the pair $(G_X,\widetilde{G_2})$ defines a path in $P$ from $p$ to $p'$.
\end{proof}

\begin{lem}\label{lem:equivalence}
Let $X$ and $Y$ be two connected topological spaces, and $i:X\to Y$ and $q:Y\to X$ be homotopy inverses of each other. For $x\in X$, let $H:[0,1]\rightarrow Y$ be a path in $Y$, such that
\[
H(0)=i(x),\quad H(1)=i\circ q\circ i(x)
\]
Then, there exists a path $T:[0,1]\rightarrow X$ such that
\[
T(0)=x, \quad T(1)=q\circ i(x)
\]
and $i\circ T$ is path homotopic to $H$ in $Y$.
\end{lem}
\begin{proof}
Since $q\circ i\sim_h \text{id}_X$, there is a path $S:[0,1]\to X$ such that $S(0)=q\circ i(x)$, and $S(1)=x$. Thus, $H\bullet (i\circ S)$ is a loop in $Y$ based at $i(x)$. Since $\pi_1(Y)=i_{\ast}(\pi_1(X))$, there exists a loop $L$ based at $x$ in $X$ such that 
\[
i_{\ast}[L]=[H\bullet (i\circ S)].
\]
Then, $T:=L\bullet\overline{S}$ is the required path.
\end{proof}

Note that, if $B$ is a C*-algebra, then the rationalization $\U_0(B)_{\Q}$ of $\U_0(B)$ carries an $H$-space structure by \cref{prop: h_space_rationalization}. We shall use $\rho^B$ to denote this multiplication map. Furthermore, we write $e_{\Q}$ and $e^2_{\Q}$ for the units of $\U_0(B)_{\Q}$ and $\U_0(M_2(B))_{\Q}$ respectively.

\begin{prop}\label{prop:homotopy}
Let $B$ be a rationally $K$-stable C*-algebra, $[f] \in \pi_j(\U_0(B))$ and $n\in \N$ such that $[\iota\circ f]$ is is an element of order $n$ in $\pi_j(\U_0(M_2(B))$. Let $H:[0,1]\to C_{\ast}(S^j,\U_0(M_2(B)))$ be a path satisfying
\[
H(0) = 0\text{ and } H(1) = \mu^{M_2(B)}_n(\iota\circ f)=\iota\circ \mu^B_n(f).
\]
Then, there exists a natural number $N\in \N$ and a path $H':[0,1]\to C_{\ast}(S^j,\U_0(B))$ such that
\[
\mu^{M_2(B)}_N(H)\sim_h \iota\circ H'
\]
in $C_{\ast}(S^j,\U_0(M_2(B)))$.
\end{prop}
\begin{proof}
Since $B$ is rationally $K$-stable, there are maps $\U_0(B)_{\Q}\to \U_0(M_2(B))_{\Q}$ and $\U_0(M_2(B))_{\Q}\to \U_0(B)_{\Q}$ which are homotopy inverses of each other. Therefore, we get a commuting diagram
\[
\xymatrix{
C_{\ast}(S^j,\U_0(B))\ar[r]^{\iota}\ar[d]_{r} & C_{\ast}(S^j,\U_0(M_2(B))\ar[d]^{R} \\
C_{\ast}(S^j,\U_0(B)_{\Q})\ar[r]^{i} & C_{\ast}(S^j,\U_0(M_2(B))_{\Q}) \\
}
\]
where $r$ and $R$ represent the rationalization maps. Furthermore, $i$ has a homotopy inverse $q:C_{\ast}(S^j,\U_0(M_2(B))_{\Q})\to C_{\ast}(S^j,\U_0(B)_{\Q})$. Let $H:[0,1]\to C_{\ast}(S^j,\U_0(M_2(B)))$ as above. Since $R$ is a rationalization map, applying \cref{lem:rational_homotopy}, we get a homotopy $G:[0,1]\to C_{\ast}(S^j,\U_0(M_2(B))_{\Q})$ such that
\[
G(0) = e^2_{\Q}, \text{ and } G(1) = R\circ \iota\circ f = \iota\circ r\circ f
\]
Furthermore, $\rho^{M_2(B)}_n(G)$ is path homotopic to $R\circ H$ in $C_{\ast}(S^j,\U_0(M_2(B))_{\Q})$. Now, $q\circ G:[0,1]\to C_{\ast}(S^j,\U_0(B)_{\Q})$ is such that
\[
q\circ G(0) = e_{\Q}, \text{ and } q\circ G(1)  = q\circ i\circ r\circ f
\]
Note that $i$ and $q$ are homotopy equivalences, hence $G$ and $i\circ q\circ G$ are homotopic in $C_{\ast}(S^j,\U_0(M_2(B))_{\Q})$ say by $K:[0,1]\times [0,1]\to C_{\ast}(S^j,\U_0(M_2(B))_{\Q})$ satisfying
\[
K(s,0) = G(s), \quad K(s,1) = i\circ q\circ G(s), \quad K(0,t) = e^2_{\Q}
\]
Define $T:[0,1]\to C_{\ast}(S^j,\U_0(M_2(B))_{\Q})$ as $T(t) = K(1,1-t)$. Then
\[
T(0) = i\circ q\circ i\circ r\circ f, \quad T(1) = i\circ r\circ f
\]
Thus, by \cref{lem:equivalence}, there is a homotopy $S:[0,1]\to C_{\ast}(S^j,\U_0(B)_{\Q})$ such that
\[
S(0) = q\circ i\circ r\circ f, \quad S(1) = r\circ f
\]
and $i\circ S$ is path homotopic to $T$ in $C_{\ast}(S^j,\U_0(M_2(B))_{\Q})$. Since $(i \circ q\circ G)\bullet T$ is path homotopic to $G$, this implies $(i\circ q\circ G)\bullet (i\circ S)$ is path homotopic to $G$ in $C_{\ast}(S^j,\U_0(M_2(B))_{\Q})$. Thus, we get a path $(q\circ G)\bullet S:[0,1]\to C_{\ast}(S^j,\U_0(B))_{\Q})$ so that 
\[
(q\circ G)\bullet S(0)= e_{\Q},\quad q\circ G\bullet S(1)=r\circ f,\quad i\circ( q\circ G\bullet S)\sim_h G
\]
Again, since $r$ is a rationalization map, by \cref{lem:rational_homotopy}, there exists a natural number $m\in\N$ and a path $H':[0,1]\to C_{\ast}(S^j,\U_0(B))$ such that 
\[
H'(0)=0,\quad H'(1)= \mu^B_m(f), \quad r\circ H'\sim_h \rho^B_m((q\circ G)\bullet S).
\]
Take $k=\text{lcm}\{n,m\}$, and write $k = n\ell_1=m\ell_2$ for some $\ell_1, \ell_2\in\N$. Then $ \mu^B_{l_2}(H'):[0,1]\to C_{\ast}(S^j,\U_0(B))$ is such that
\[
\mu^B_{\ell_2}(H')(0)=0, \mu^B_{\ell_2}(H')(1)= \mu^B_k(f), \text{ and } r\circ\mu^B_{\ell_2}( H')\sim_h \rho^B_k((q\circ G)\bullet S).
\]
Also $ \mu^{M_2(B)}_{\ell_1}(H) :[0,1]\to C_{\ast}(S^j,\U_0(M_2(B))$ is such that
\[
\mu^{M_2(B)}_{\ell_1}(H)(0)=0, \mu^{M_2(B)}_{\ell_1}( H)(1)=\iota\circ \mu^B_k(f), \text{ and } R\circ \mu^{M_2(B)}_{\ell_1}(H)\sim_h \rho^{M_2(B)}_k(G).
\]
Then, from the earlier arguments, we have the following relations
\begin{equation*}
\begin{split}
 \rho^{M_2(B)}_{k}(G) &\sim_h R\circ\mu^{M_2(B)}_{\ell_1}(H), \\
i\circ((q\circ G)\bullet S) &\sim_h G, \text{ and } \\
r\circ\mu^B_{\ell_2}( H') &\sim_h \rho^B_k((q\circ G)\bullet S)
\end{split}
\end{equation*}
Also $i\circ r\circ H'=R\circ \iota\circ H'$. Hence 
\begin{equation*}
\begin{split}
R\circ \iota\circ\mu^B_{\ell_2}(H') &= i\circ r\circ \mu^B_{\ell_2}(H') \sim_h  \rho^{M_2(B)}_k(i\circ((q\circ G)\bullet S)) \\
&\sim_h  \rho^{M_2(B)}_k(G)\sim_h R\circ\mu^{M_2(B)}_{\ell_1}(H)
\end{split}
\end{equation*}
Thus 
\[
\left[R\circ\left(\iota\circ \mu^B_{\ell_2}( H')\bullet \overline{\mu^{M_2(B)}_{\ell_1}( H)}\right)\right]=0
\]
in $\pi_1(C_{\ast}(S^j,\U_0(M_2(B))_{\Q})$. Then, by \cref{lem:rational_homotopy}, there exists a natural number $P\in\N$ such that
\[
\iota\circ\mu^{B}_{P\ell_2}(H')=\mu^{M_2(B)}_P(\iota\circ \mu^B_{\ell_2}( H'))\sim_h \mu^{M_2(B)}_P( \mu^{M_2(B)}_{\ell_1}(H))=\mu^{M_2(B)}_{P\ell_1}(H)
\]
in $C_{\ast}(S^1,\U_0(M_2(B)))$. Thus, replacing $H'$ by $ \mu^B_{P \ell_2}(H')$ and taking $N:=P \ell_1$, we have
\[
\iota\circ H'\sim_h  \mu^{M_2(B)}_N(H).\qedhere
\]
\end{proof}
 
The next lemma is an analogue of \cite[Lemma 2.7]{vaidyanathan}, and is a consequence of that result and \cref{prop:homotopy}.

\begin{lem}\label{lem:local_path}
Given $X$ a compact Hausdorff space, $A$ a continuous $C(X)$-algebra, and $x\in X$ such that $A(x)$ is rationally $K$-stable. For $[f]\in \pi_j(\U_0(A))$, let $F : [0,1]\to C_*(S^j,\,\U_0(M_2(A))\,)$ be a path and $n\in\N$ such that
\[
F(0) = 0 \text{ and } F(1) = \mu^{M_2(A)}_n(\iota\circ f).
\]
Then, there is a closed neighbourhood $Y$ of $x$, a natural number $N_x\in\N$ and a path $L_Y : [0,1]\to C_*(S^j,\,\U_0(A(Y))\,)$ such that $L_Y(0) = 0, L_Y(1) = \mu^{A(Y)}_{N_xn}(\pi_Y\circ f)$, and
\[
\iota_Y\circ L_Y\sim_h \mu^{M_2(A(Y))}_{N_x}(\eta_Y\circ F)
\]
in $C_*(S^j,\,\U_0(M_2(A(Y))))$.
\end{lem}

\begin{rem}\label{rem:reorganize_cover}
We are now in a position to prove \cref{main_thm}, but first, we need one important fact, which allows us to use induction: If $X$ is a finite dimensional compact metric space, then covering dimension agrees with the small inductive dimension \cite[Theorem 1.7.7]{engelking}. Therefore, by \cite[Theorem 1.1.6]{engelking}, $X$ has an open cover $\mathcal{B}$ such that, for each $U \in \mathcal{B}$,
\[
\dim(\partial U)\leq \dim(X) - 1.
\]
Now suppose $\{U_1,U_2,\ldots, U_m\}$ is an open cover of $X$ such that $\dim(\partial U_i)\leq \dim(X)-1$ for $1\leq i\leq m$, we define sets $\{V_i : 1\leq i\leq m\}$ inductively by
\[
V_1 := \overline{U_1}, \text{ and } V_k := \overline{U_k\setminus \left( \bigcup_{i<k} U_i\right)} \text{ for } k>1
\]
and subsets $\{W_j : 1\leq j\leq m-1\}$ by
\[
W_j := \left(\bigcup_{i=1}^j V_i\right)\cap V_{j+1}.
\]
It is easy to see that $W_j \subset \bigcup_{i=1}^j \partial U_i$, so by \cite[Theorem 1.5.3]{engelking}, $\dim(W_j) \leq \dim(X)-1$ for all $1\leq j\leq m-1$.
\end{rem}

\begin{proof}[Proof of \cref{main_thm}]
Let $A$ be a continuous $C(X)$-algebra such that each fiber of $A$ is rationally $K$-stable. By \cref{thm:dim_zero_case_}, we assume that $\dim(X)\geq 1$, and we assume that $A(Y)$ is rationally $K$-stable for any closed subset $Y\subset X$ with $\dim(Y) \leq \dim(X)-1$. We now show that the map
\[
\iota_{\ast}\otimes \text{id}:\pi_j(\U_0(M_n(A))\,)\otimes\Q\to\pi_j(\U_0(M_{n+1}(A))\,)\otimes \Q 
\]
is an isomorphism for $j\geq 1$, $n\geq 1$. For simplicity of notation, we assume that $n=1$.\\

We first prove injectivity. Fix $[f]\in\pi_j(\U_0(A))$ such that $[\iota\circ f]$ has order $n$ in $\pi_j(\U_0(M_2(A)))$, then we wish to prove that $[f]$ has finite order in $\pi_j(\U_0(A))$. For this consider $F:[0,1]\to C_{\ast}(S^j,\U_0(M_2(A)))$ such that 
\[
F(0)=0,\quad F(1)=\mu^{M_2(A)}_n(\iota\circ f).
\]
For $x\in X$, by \cref{lem:local_path}, there is a closed neighbourhood $Y_x$ of $x$ , $N_x\in\N$ and a path $L_{Y_x} : [0,1]\to C_{\ast}(S^j,\U_0(A(Y_x)))$ such that
\[
L_{Y_x}(0) = 0, \quad L_{Y_x}(1) = \mu^{A(Y_x)}_{N_xn}(\pi_{Y_x}\circ f)
\]
and $\iota_{Y_x}\circ L_{Y_x}\sim_h\mu^{M_2(A(Y_x))}_{N_x}(\eta_{Y_x}\circ F)$ in $C_{\ast}(S^j,\U_0(M_2(A(Y_x))))$. We may choose $Y_x$ to be the closure of a basic open set $U_x$ such that $\dim(\partial U_x)\leq \dim(X)-1$. Since $X$ is compact, we may choose a finite subcover $\{U_1,U_2,\ldots, U_m\}$. Now define $\{V_1,V_2,\ldots, V_m\}$ and $\{W_1,W_2,\ldots, W_{m-1}\}$ as in \cref{rem:reorganize_cover}. We observe that each $V_i$ is a closed set such that $\mu^{A(V_i)}_{N_in}(\pi_{V_i}\circ f)\sim 0$ in $C_{\ast}(S^j,\U_0(A(V_i)))$ since $V_i \subset \overline{U_i}$ for all $1\leq i\leq m$. \\

Note that $W_1 = V_1\cap V_2$, and $\dim(W_1)\leq \dim(X) - 1$. By induction hypothesis, $A(W_1)$ is rationally $K$-stable. Let $H_i:[0,1]\to C_{\ast}(S^j,\U_0(A(V_i)))$, $i=1,2$ be paths such that $H_i(0) = 0, H_i(1) = \mu^{A(V_i)}_{N_in}(\pi_{V_i}\circ f)$, and 
\[
\iota_{V_i}\circ H_i \sim_h \mu^{M_2(A(V_i))}_{N_i}(\eta_{V_i}\circ F).
\]
Setting $M := \text{lcm}(N_1,N_2)$, we may assume that $H_i:[0,1]\to C_{\ast}(S^j,\U_0(A(V_i)))$, $i=1,2$ are paths such that $H_i(0) = 0, H_i(1) =\mu^{A(V_i)}_{Mn}(\pi_{V_i}\circ f)$, and
\[
\iota_{V_i}\circ H_i \sim_h \mu^{M_2(A(V_i))}_M(\eta_{V_i}\circ F).
\]
Let $S:[0,1]\to C_{\ast}(S^j,\U_0(A(W_1)))$ be the path
\[
S := (\pi^{V_1}_{W_1}\circ H_1)\bullet \overline{(\pi^{V_2}_{W_1}\circ H_2)}.
\]
Note that $S(0) = S(1) = 0$, so $S$ is a loop in $C_{\ast}(S^j,\U_0(A(W_1)))$, and 
\begin{eqnarray*}
\iota_{W_1}\circ S &=& (\eta^{V_1}_{W_1}\circ \iota_{V_1}\circ H_1)\bullet (\eta^{V_2}_{W_1}\circ \iota_{V_2}\circ \overline{H_2})\\ &\sim_h & \mu^{M_2(A(W_1))}_M(\eta_{W_1}\circ F\bullet \overline{(\eta_{W_1}\circ F})) \sim_h 0.
\end{eqnarray*}
Also, since $A(W_1)$ is rationally $K$-stable \[\iota_{W_1}^{\ast}\otimes \text{id}:\pi_1(C_{\ast}(S^j,\U_0(A(W_1))))\otimes\Q\to \pi_1(C_{\ast}(S^j,\U_0(M_2(A(W_1)))\,)\otimes\Q\] is an isomorphism. Hence, there exists $m\in\N$ such that $m[S]=0$ in $\pi_1(C_{\ast}(S^j,\U_0(A(W_1))))$. Thus 
\[
\pi^{V_1}_{W_1}\circ\mu^{A(V_1)}_m(H_1)=\mu^{A(W_1)}_m(\pi^{V_1}_{W_1}\circ H_1)\sim_h \mu^{A(W_1)}_m(\pi^{V_2}_{W_1}\circ H_2)=\pi^{V_2}_{W_1}\circ\mu^{A(V_2)}_m(H_2)
\] in $C_{\ast}(S^j,\U_0(A(W_1)))$. Now, by \cref{lem:cx_algebra_pullback}, and \cite[Theorem 3.9]{pedersen}, we have a pullback diagram
\[
\xymatrixcolsep{1.5cm}\xymatrix{
C_{\ast}(S^j, A(V_1\cup V_2))\ar[r]^{\pi^{V_1\cup V_2}_{V_1}}\ar[d]_{\pi^{V_1\cup V_2}_{V_2}} & C_{\ast}(S^j,A(V_1))\ar[d]^{\pi^{V_1}_{W_1}} \\
C_{\ast}(S^j,A(V_2))\ar[r]^{\pi^{V_2}_{W_1}} & C_{\ast}(S^j,A(W_1)).
}
\]
As mentioned before, this induces a pullback diagram of groups of quasi-unitaries. Furthermore, the map $\pi^{V_1}_{W_1}: \U_0(A(V_1)) \to \U_0(A(W_1))$ is a Serre fibration. Thus, by \cref{hom_pullback},
\[
\mu^{A(V_1\cup V_2)}_{mMn}(\pi_{V_1\cup V_2}\circ f)\sim_h 0
\]
in $C_*(S^j,\U_0(A(V_1\cup V_2)))$. Thus, $mMn[\pi_{V_1\cup V_2}\circ f]=0$, so that $[\pi_{V_1\cup V_2}\circ f]$ has finite order in $\pi_j(\U_0(A(V_1\cup V_2)))$. \\

Now observe that $W_2 = (V_1\cup V_2)\cap V_3$, and $\dim(W_2) \leq \dim(X)-1$. Replacing $V_1$ by $V_1\cup V_2$ and $V_2$ by $V_3$ in the above argument, we may repeat the earlier procedure. By induction on the number of elements in the finite subcover, we conclude that $[f]$ has finite order in $\pi_j(\U_0(A))$, as required.\\

We now prove surjectivity of $\iota_{\ast}\otimes \text{id}$. Choose $[u] \in \pi_j(\U_0(M_2(A)))$ and $m\in \Z$ non-zero. We wish to construct an element $[\omega]\in\pi_j(\U_0(A))$ and $q\in \Q$ such that
\[
\iota_{\ast}\otimes \text{id}\left([\omega]\otimes q\right)=[u]\otimes\frac{1}{m}.
\]
So, fix $x\in X$. Then by rationally $K$-stability of $A(x)$ (as in the proof of \cref{thm:dim_zero_case_}), there is a closed neighbourhood $Y_x$ of $x$, a natural number $L_x\in\N$, and a quasi-unitary $c_x \in C_{\ast}(S^j,\U_0(A(Y_x))$ such that
\[
\mu^{M_2(A(Y_x))}_{L_x}(\eta_{Y_x}\circ u) \sim_h \iota_{Y_x}\circ c_x.
\]
As in the first part of the proof, we may reduce to the case where $X = V_1\cup V_2$, and there are quasi-unitaries $c_{V_1} \in C_{\ast}(S^j,\U_0(A(V_1))), c_{V_2} \in C_{\ast}(S^j,\U_0(A(V_2)))$ such that
\[
\mu^{M_2(A(V_i))}_{L_i}(\eta_{V_i}\circ u) \sim \iota_{V_i}\circ c_{V_i} \text{ in } C_{\ast}(S^j,\U_0(M_2(A(V_i))))\quad i=1,2
\]
and if $W := V_1\cap V_2$, then $\dim(W)\leq \dim(X)-1$. Furthermore, by replacing the $\{L_i\}$ by their least common multiple, we may assume that $L_1 = L_2 =: L$. Now, fix paths $H_i : [0,1]\to C_{\ast}(S^j,\U_0(M_2(A(V_i))))$ such that
\begin{equation*}
\begin{split}
H_1(0) = \iota_{V_1}\circ c_{V_1},\qquad & H_1(1) = \mu^{M_2(A(V_1))}_L(\eta_{V_1}\circ u)\\
H_2(0)=\mu^{M_2(A(V_2))}_L(\eta_{V_2}\circ u), \qquad & H_2(1)=\iota_{V_2}\circ c_{V_2}.
\end{split}
\end{equation*}
Consider the path $F :[0,1]\to C_{\ast}(S^j,\U_0(M_2(A(W))))$ given by
\[
F := (\eta^{V_1}_W\circ H_1) \bullet (\eta^{V_2}_W\circ H_2).
\]
Then $F(0) = \iota_W\circ\pi^{V_1}_W\circ c_{V_1}$ and $F(1) =\iota_W\circ\pi^{V_2}_W\circ c_{V_2}$. Then since $A(W)$ is rationally $K$-stable, by \cref{prop:homotopy}, there exists a path $F':[0,1] \to C_{\ast}(S^j,\U_0(A(W)))$ and a natural number $N\in \N$ such that
\[
F'(0)=\mu^{A(W)}_N(\pi^{V_1}_W\circ c_{V_1}),\quad F'(1)=\mu^{A(W)}_N(\pi^{V_2}_W\circ c_{V_2})
\]
and $\iota_W\circ F'$ is path homotopic to $\mu^{M_2(A(W))}_N (F)$ in $C_{\ast}(S^j,\U_0(M_2(A(W))))$. The map $\pi^{V_2}_W : C_{\ast}(S^j,\U_0(A(V_2))) \to \pi^{V_2}_W(C_{\ast}(S^j,\U_0(A(V_2))))$ is a fibration, so there is a path $F'' : [0,1]\to C_{\ast}(S^j,\U_0(A(V_2)))$ such that
\[
F''(1) =\mu^{A(V_2)}_N(c_{V_2}), \text{ and } \pi^{V_2}_W\circ F'' = F'.
\]
Define $e_{V_2} := F''(0)$ so that
\[
\pi^{V_2}_W\circ e_{V_2} = \mu^{A(W)}_N(\pi^{V_1}_W\circ c_{V_1}).
\]
Recall that, given a path $G$ in a topological space, the path $\widetilde{G}$ is defined by \cref{eqn: tilde_path}. Define $H_3:[0,1]\to C_*(S^j,\U_0(M_2(V_2)))$ as
\[
H_3:=\mu^{M_2(A(V_2))}_N(H_2)\bullet \widetilde{(\iota_{V_2}\circ \overline{F''})}.
\]
Then, $H_3(0)=\mu^{M_2(A(V_2))}_{NL}(\eta_{V_2}\circ u), H_3(1)=\iota_{V_2}\circ e_{V_2}$, and
\[
\eta^{V_2}_W\circ H_3=\eta^{V_2}_W\circ(\mu^{M_2(A(V_2))}_N( H_2))\bullet\widetilde{( \iota_W\circ \overline{F'})}.
\]
Also $\eta^{V_1}_W : C_{\ast}(S^j,\U_0(M_2(A(V_1)))) \to \eta^{V_1}_W(C_{\ast}(S^j,\U_0M_2((A(V_1))))$ is a fibration, thus $\eta^{V_2}_W\circ(\mu^{M_2(A(V_2))}_N( H_2))$ has a lift, denoted by $T:[0,1]\to C_*(S^j,\U_0(M_2(A(V_1))))$ so that 
\[
T(0)=\mu^{M_2(A(V_1))}_{NL}(\eta_{V_1}\circ u).
\]
Then, letting $G:=\mu^{M_{2(A(V_1))}}_N(H_1)\bullet T$ gives $\eta^{V_1}_W\circ G=\mu^{M_2(A(W))}_N(F)$. Again by the above fibration map, since  $\eta^{V_1}_W\circ G=\mu^{M_2(A(W))}_N(F)\sim_h \iota_W\circ F'$, by calculation done in \cref{hom_pullback}, $\widetilde{\iota_W\circ F'}$ has a lift in $C_*(S^j,\U_0(M_2(V_1)))$, denoted by $T'$. Then
\[
\eta^{V_1}_W\circ(T\bullet \overline{T'})=\eta^{V_2}_W\circ(\mu^{M_2(A(V_2))}_N(H_2))\bullet\widetilde{(\iota_W\circ\overline{F'})}.
\]
As before, $C_{\ast}(S^j,A)$ is a pullback
\[
\xymatrixcolsep{1.5cm}\xymatrix{
C_{\ast}(S^j,A)\ar[r]^{\pi{_{V_1}}}\ar[d]_{\pi_{V_2}} & C_{\ast}(S^j, A(V_1))\ar[d]^{\pi^{V_1}_W} \\
C_{\ast}(S^j,A(V_2))\ar[r]^{\pi^{V_2}_W} & C_{\ast}(S^j,A(W))
}
\]
so that $\omega := (\mu^{A(V_1)}_N(c_{V_1}), e_{V_2})$ defines a quasi-unitary in $C_{\ast}(S^j,A)$, and $\iota\circ \omega\sim \mu^{M_2(A)}_{NL} (u)$ in $C_{\ast}(S^j,\U_0(M_2(A)))$, where the path is given by the pair $(H_3, T\bullet \overline{T'})$. Hence, for $q := 1/(mNL)$, we have
\[
\iota_{\ast}\otimes \text{id}([\omega]\otimes q) = [u]\otimes \frac{1}{m}
\]
as required.
\end{proof}

We conclude this section with a discussion on the extent to which the converse of \cref{main_thm} holds.

\begin{prop}\label{trivial_rational}
Let $X$ be a locally compact, Hausdorff space, and $A$ be a $C^*$-algebra. If $A$ is rationally $K$-stable, then so is $C_0(X)\otimes A$. The converse is true if $X$ is a finite CW-complex.
\end{prop} 
\begin{proof}
If $A$ is rationally $K$-stable, we wish to show that $C_0(X)\otimes A$ is rationally $K$-stable. By appealing to the five lemma (as in \cite[Lemma 2.1]{vaidyanathan}), we may assume that $X$ is compact. Now, $X$ is an inverse limit of compact metric spaces $(X_i)$ by \cite{mardesic}, so that $C(X)\otimes A \cong \lim C(X_i)\otimes A$. Since the functors $F_j$ are continuous (\cref{prop: continuous_homology}), we may assume that $X$ itself is a compact metric space. Any metric space can, in turn, be written as an inverse limit of finite CW-complexes \cite{freudenthal}. Therefore, we may further assume that $X$ is a finite CW-complex. In that case, by \cite[Theorem 4.20]{lupton}, one has
\begin{equation}\label{eqn: lupton}
F_j(C(X,A)) \cong \bigoplus_{n\geq j}H^{n-j}(X;F_j(A))
\end{equation}
where the isomorphism is natural. Since the map $\iota_{\ast} :F_j(M_{n-1}(A))\to F_j(M_n(A))$ is an isomorphism, it follows that $\iota_{\ast} : F_j(C(X,M_{n-1}(A)))\to F_j(C(X,M_n(A)))$ is an isomorphism as well. Hence, $C(X)\otimes A$ is rationally $K$-stable. \\

Now suppose $X$ is a finite CW-complex and $C(X)\otimes A$ is $K$-stable. Then,
\[
F_j(C(X,M_{n-1}(A)))\cong F_j(C(X,M_{n}(A)))
\]
and the isomorphism of \cref{eqn: lupton} is component-wise. This implies that
\[
H^{n-j}(X;F_j(M_{n-1}(A)))\cong H^{n-j}(X;F_j(M_n(A))),\quad \forall n\geq j 
\] 
For any connected $H$-space $Y$, as in \cref{ex: rat_k_stable_not_k_stable}, there is a fibration sequence $C_{\ast}(X,Y)\to C(X,Y)\to Y$, which induces a short exact sequence of rational homotopy groups
\begin{equation}\label{eqn: ses_rat_hom_h_space}
0 \to F_j(C_{\ast}(X,Y)) \to F_j(C(X,Y)) \to F_j(Y) \to 0
\end{equation}
Now, we take $Y=\U_0(M_k(A))$ and apply \cite[Theorem 4.20]{lupton} to get
\[
F_j(C_{\ast}(X,M_k(A))) \cong \bigoplus_{n\geq j}\widetilde{H}^{n-j}(X;F_j(M_k(A)))
\]
and the isomorphism is natural. Hence, we conclude that
\[
F_j(C_{\ast}(X,M_{n-1}(A)))\cong F_j(C_{\ast}(X,M_{n}(A)))
\]
as well. By \cref{eqn: ses_rat_hom_h_space} and the five lemma, we conclude that $A$ is rational $K$-stable.
\end{proof}

In \cite[Theorem B]{seth}, we proved that, for an AF-algebra, rational $K$-stability is equivalent to $K$-stability. Combining this fact with \cref{trivial_rational}, and \cite[Theorem A]{vaidyanathan}, we have 

\begin{cor}
Let $X$ be a finite CW-complex, and $A$ be an $AF$-algebra. Then, $C(X)\otimes A$ is $K$-stable if and only if $A$ is $K$-stable.
\end{cor}

The next example shows that the converse of \cref{main_thm} need not hold for arbitrary continuous $C(X)$-algebras
\begin{ex}
Let $D_1 := M_{2^{\infty}}$ denote the UHF algebra of type $2^{\infty}$, and let $D_2 := D_1\oplus M_2(\C)$. Consider the $C[0,1]$-algebra
\[
A :=  \{(f,g) \in C[0,1/2]\otimes D_1\oplus C[1/2,1]\otimes D_2 : \Phi(f)=g(1/2)\}
\]
where $\Phi:C[0,1/2]\otimes D_1\to D_2$ is given by $\Phi(f)=(f(1/2),0)$. Since $\Phi$ is injective, it follows that $A$ is a continuous $C[0,1]$-algebra. Note that $A$ may be described as a pullback
\[
\xymatrix{
A\ar[r]\ar[d] & C[\frac{1}{2},1]\otimes D_2\ar[d]^{\text{ev}} \\
C[0,\frac{1}{2}]\otimes D_1 \ar[r]_{\Phi} & D_2
}
\]
where $\text{ev}$ is the evaluation at $1/2$. The Mayer-Vietoris theorem \cite[Theorem 4.5]{schochet} for the functor $F_m$ gives a long exact sequence
\[
\ldots \to F_m(A) \to F_m(D_2)\oplus F_m(D_1)\xrightarrow{\text{ev}_{\ast}-\Phi_{\ast}} F_m(D_2)\to \ldots
\]
where $\text{ev}_{\ast} :F_m(D_2)\to F_m(D_2)$ is the identity map and $\Phi_{\ast} :F_m(D_1)\to F_m(D_2)$ is given as $\Phi_{\ast}(r)=(r,0)$, thus $(\text{ev}_{\ast}-\Phi_{\ast}):F_m(D_2)\oplus F_m(D_1)\to F_m(D_2)$ is given by 
\[
(\text{ev}_{\ast}-\Phi_{\ast})((a,b),c)=(a-c,b).
\]
Consider the case where $m$ is odd: By \cite[Lemma 3.2]{seth}, $F_{m-1}(D_i) = F_{m+1}(D_i) = 0$ for $i=1,2$. Hence, the above long exact sequence boils down to
\[
0 \to F_m(A) \to F_m(D_2)\oplus F_m(D_1)\xrightarrow{\text{ev}_{\ast}-\Phi_{\ast}} F_m(D_2)\to \ldots
\]
Thus, there is a natural isomorphism
\[
F_m(A) = \ker(\text{ev}_{\ast}-\Phi_{\ast}) \cong F_m(D_1)
\]
Similarly, $F_m(M_2(A)) \cong F_m(M_2(D_1))$ and the following diagram commutes
\[
\xymatrix{
F_m(A)\ar[r]^{\cong}\ar[d]_{\iota^A} & F_m(D_1)\ar[d]_{\iota^{D_1}} \\
F_m(M_2(A))\ar[r]^{\cong} & F_m(M_2(D_1)).
}
\]
Since $D_1$ is rationally $K$-stable by \cite[Theorem B]{seth}, it follows that $\iota^A$ is an isomorphism. Doing the same for the inclusion map $M_n(A)\hookrightarrow M_{n+1}(A)$, we conclude that the map $F_m(M_n(A)) \to F_m(M_{n+1}(A))$ is an isomorphism if $m$ is odd. \\

Now suppose $m$ is even: The above long exact sequence reduces to 
\[
F_{m-1}(A)\to F_{m-1}(D_2)\oplus F_{m-1}(D_1)\xrightarrow{\text{ev}_{\ast}-\Phi_{\ast}} F_{m-1}(D_2)\to F_m(A)\to 0
\]
so that $F_m(A)\cong \text{coker}(\text{ev}_{\ast}-\Phi_{\ast})$. Now, by \cite[Theorem A]{seth}, it follows that $F_{m-1}(D_1)\cong \Q$ for all even $m$, and
\[
F_{m-1}(D_2) \cong \begin{cases}
\Q\oplus \Q &: m=2,4 \\
\Q &: m>4 \text{ even}.
\end{cases}
\]
Thus, elementary linear algebra proves that $\text{ev}_{\ast}-\Phi_{\ast}$ is surjective, so that $F_m(A) = 0$. Similarly, $F_m(M_n(A)) = 0$ for all $n\geq 2$ as well (if $m$ is even). \\

Thus, we conclude that $A$ is rationally $K$-stable. However, one of its fibers (namely $D_2$) is not rationally $K$-stable because it has a non-zero finite dimensional representation \cite[Theorem B]{seth}.
\end{ex}

\section{An application to Crossed Product C*-algebras}\label{sec: rokhlin}

As an application of our earlier results, we wish to show that the class of (rationally) $K$-stable C*-algebras is closed under the formation of certain crossed products. To begin with, we fix some conventions. In what follows, $G$ will denote a compact, second countable group, and $A$ will denote a separable C*-algebra. By an action of $G$ of $A$, we mean a continuous group homomorphism $\alpha : G\to \text{Aut}(A)$, where $\text{Aut}(A)$ is equipped with the point-norm topology. We write $\sigma: G\to \text{Aut}(C(G))$ for the left action of $G$ on $C(G)$, given by $\sigma_s(f)(t) := f(s^{-1}t)$. \\

The notion of Rokhlin dimension was invented by Hirshberg, Winter and Zacharias \cite{hwz} for actions of finite groups (and the integers). The definition for compact, second countable groups is due to Gardella \cite{gardella_compact}. The `local' definition we give below is different from the original, but is equivalent due to \cite[Lemma 3.7]{gardella_compact} (See also \cite[Lemma 1.5]{prahlad_rokhlin}).

\begin{defn}\label{defn: rokhlin_dimension}
Let $G$ be a compact, second countable group, and let $A$ be a separable C*-algebra. We say that an action $\alpha:G\to \text{Aut}(A)$ has Rokhlin dimension $d$ (with commuting towers) if $d$ is the least integer such that, for any pair of finite sets $F\subset A, K\subset C(G)$, and any $\epsilon > 0$, there exist $(d+1)$ contractive, completely positive maps
\[
\psi_0, \psi_1, \ldots, \psi_d : C(G)\to A
\]
satisfying the following conditions:
\begin{enumerate}
\item For $f_1,f_2\in K$ such that $f_1\perp f_2$, $\|\psi_j(f_1)\psi_j(f_2)\| < \epsilon$ for all $0\leq j\leq d$.
\item For any $a\in F$ and $f\in K$, $\|[\psi_j(f), a]\| < \epsilon$ for all $0\leq j\leq d$.
\item For any $f\in K$ and $s\in G, \|\alpha_s(\psi_j(f)) - \psi_j(\sigma_s(f))\| < \epsilon$ for all $0\leq j\leq d$.
\item For any $a\in F, \|\sum_{j=0}^d \psi_j(1_{C(G)})a - a\| < \epsilon$.
\item For any $f_1,f_2 \in K$, $\|[\psi_j(f_1),\psi_k(f_2)]\| < \epsilon$ for all $0\leq j,k\leq d$.
\end{enumerate}
We denote the Rokhlin dimension (with commuting towers) of $\alpha$ by $\dim_{Rok}^c(\alpha)$. If no such integer exists, we say that $\alpha$ has infinite Rokhlin dimension (with commuting towers), and write $\dim_{Rok}^c(\alpha) = +\infty$. 
\end{defn}

We now describe the local approximation theorem due to Gardella, Hirshberg and Santiago \cite{gardella} that will help prove the permanence result we are interested in.

\begin{prop}\label{prop: target_field}\cite[Corollary 4.9]{gardella}
Let $G$ be a compact, second countable group, $X$ be a compact Hausdorff space and $A$ be a separable C*-algebra. Let $G\curvearrowright X$ be a continuous, free action of $G$ on $X$, and $\alpha:G\to \text{Aut}(A)$ be an action of $G$ on $A$. Equip the C*-algebra $C(X,A)$ with the diagonal action of $G$, denoted by $\gamma$. Then, the crossed product C*-algebra $C(X,A)\rtimes_{\gamma} G$ is a continuous $C(X/G)$-algebra, each of whose fibers are isomorphic to $A\otimes \mathcal{K}(L^2(G))$.
\end{prop}

In the context of \cref{prop: target_field}, the natural inclusion map $\rho : A\to C(X,A)$ is a $G$-equivariant $\ast$-homomorphism. Hence, it induces a map $\rho : A\rtimes_{\alpha} G \to C(X,A)\rtimes_{\gamma} G$. To describe the nature of this map, we need the next definition, which is due to Barlak and Szabo \cite{barlak}. Once again, we choose to work with the local definition as it is more convenient for our purpose.

\begin{defn}\label{defn: seq_split}
Let $A$ and $B$ be separable C*-algebras. A $\ast$-homomorphism $\varphi: A\rightarrow B$ is said to be \emph{sequentially split} if, for every compact set $F\subset A$, and for every $\epsilon >0$, there exists a $\ast$-homomorphism $\psi = \psi_{F,\epsilon}: B\rightarrow A$ such that 
\[
\|\psi\circ \phi(a)-a\|<\epsilon
\]
for all $a\in F$.
\end{defn}

The next theorem, due to Gardella, Hirshberg and Santiago \cite[Proposition 4.11]{gardella} is an important structure theorem that allows one to prove permanence results concerning crossed products with finite Rokhlin dimension (with commuting towers).

\begin{theorem}\label{thm: seq_split_rokhlin_dimension}
Let $\alpha:G\to \text{Aut}(A)$ be an action of a compact, second countable group on a separable C*-algebra such that $\dim_{Rok}^c(\alpha) < \infty$. Then, there is exists a compact metric space $X$ and a free action $G\curvearrowright X$ such that the canonical embedding
\[
\rho : A\rtimes_{\alpha} G\to C(X,A)\rtimes_{\gamma} G
\]
is sequentially split. Furthermore, if $G$ finite dimensional, then $X$ may be chosen to be finite dimensional as well.
\end{theorem}

In light of \cref{thm: seq_split_rokhlin_dimension}, we now show that the property of being rationally $K$-stable ($K$-stable) passes from the target algebra $B$ to the domain algebra $A$, in the presence of a sequentially split $\ast$-homomorphism. To this end, we fix the following notations, given $\ast$-homomorphism $\varphi:A\rightarrow B$, $\varphi_n:M_n(A)\rightarrow M_n(B)$ represents the inflation of $\varphi$, given by $\varphi_n((a_{i,j}))= (\varphi(a_{i,j}))$. Furthermore $\iota^B:B\rightarrow M_2(B)$ represents the canonical inclusion.

\begin{prop}\label{prop: sequentially_split}
Let $A$ and $B$ be separable C*-algebras, and $\varphi: A\rightarrow B$ be a sequentially split $\ast$-homomorphism. If $B$ is rationally $K$-stable ($K$-stable), then so is $A$.
\end{prop}   
\begin{proof}
Since the proof of both cases is entirely similar, we only prove that rational $K$-stability passes from $B$ to $A$. As before, we need to show that the map
\[
(\iota^A)_{\ast}\otimes \text{id}:\pi_j(\U_0(M_n(A)))\otimes\Q\to\pi_j(\U_0(M_{n+1}(A)))\otimes \Q 
\]
is an isomorphism for all $j\geq 1$, and $n\geq 1$. If $\varphi : A\to B$ is sequentially split, then so is $\varphi_n : M_n(A)\to M_n(B)$, so we may assume without loss of generality that $n=1$. \\

We first show that $(\iota^A)_{\ast}\otimes \text{id}$ is injective. So suppose $[f]\otimes q\in \pi_j(\U_0(A))\otimes \Q$ is such that $[\iota^A\circ f]\otimes q=0$ in $\pi_j(\U_0(M_2(A)))$. Then, $[\iota^A\circ f]$ has finite order in $\pi_j(\U_0(M_2(A)))$, which implies $[\varphi_2\circ\iota^A\circ f]=[\iota^B\circ\varphi\circ f]$ has finite order in $\pi_j(\U_0(M_2(B))\,)$. Since $B$ is rationally $K$-stable, $[\varphi\circ f]$ also has finite order in $\pi_j(\U_0(B))$. Let $F := \{f(x):\, x\in S^j\}$, which is a compact set in $A$, so there exists a $\ast$-homomorphism $\psi = \psi_{F,1}:B\rightarrow A$ such that $\|\psi \circ\varphi(a)-a\|<1$ for all $a\in F$. Hence,
\[
\|\psi\circ\varphi\circ f-f\|<1
\]
in $\U(C_{\ast}(S^j,A))$. Thus, by \cref{lem:quasi_homotopy}, we conclude that
\[
[\psi \circ\varphi\circ f]=[f]
\]
in $\pi_j(\U_0(A))$. However, since $[\varphi\circ f]$ has finite order in $\pi_j(\U_0(B))$, $[\psi\circ\varphi\circ f]=[f]$ has finite order in $\pi_j(\U_0(A))$. Hence, $[f]\otimes q = 0$, proving that $(\iota^A)_{\ast}\otimes \text{id}$ is injective.\\

For surjectivity, fix an element $[u] \in \pi_j(\U_0(M_2(A))$ and $m\in \Z$. We wish to construct elements $[\omega]\in \pi_j(\U_0(A))$ and $q \in \Q$ such that 
\[
((\iota^A)_{\ast}\otimes \text{id})\left([\omega]\otimes q\right)=[u]\otimes\frac{1}{m}.
\]
Now, $[\varphi_2\circ u]\otimes\frac{1}{m}\in\pi_j(\U_0(M_2(B)))\otimes\Q$. Since $B$ is rationally $K$-stable, there exists $[g]\in\pi_j(\U_0(B))$ and $n\in \Z$ such that
\[
((\iota^B)_{\ast}\otimes\text{id})\left([g]\otimes\frac{1}{n}\right)=[\varphi_2\circ u]\otimes\frac{1}{m}.
\]
Again, as in previous calculations, there exists $N_1,N_2\in \N$ such that
\begin{equation}\label{equal}
N_1[\iota^B\circ g]=N_2[\varphi_2\circ u]
\end{equation}
in $\pi_j(\U_0(M_2(B)))$. Now, fix $F :=\{u(x): x\in S^j\}$, so we get a $\ast$-homomorphism $\psi_F:M_2(B)\rightarrow M_2(A)$ such that $\|\psi_F\circ\varphi_2\circ u(x)-u(x)\|<\frac{1}{2}$ for all $x\in S^j$. Hence,
\begin{equation}\label{equality}
[\psi_F\circ\phi_2\circ u]=[u]
\end{equation}
in $\pi_j(\U_0(M_2(A)))$. Now, we write $u = (u_{i,j})_{1\leq i,j\leq 2}$, and take $K=\{u_{i,j}(x): 1\leq i,j\leq 2, x\in S^j\}\subset A$. Then $K$ is compact, so we get a $\ast$-homomorphism $\psi_K:B\rightarrow A$ such that 
\[
\|\psi_K\circ\varphi\circ u_{i,j}(x)-u_{i,j}(x)\|<\frac{1}{8}
\]
for all $x\in S^j$ and $1\leq i,j\leq 2$. Thus $\|(\psi_K)_2\circ\varphi_2\circ u(x)-u(x)\|<\frac{1}{2}$ for all $x\in S^j$. Therefore, $\|\psi_F\circ\varphi_2\circ u(x)-(\psi_K)_2\circ\phi_2\circ u(x)\|<1$ for all $x\in S^j$, so that $[\psi_F\circ\varphi_2\circ u]=[(\psi_K)_2\circ\varphi_2\circ u]$ in $\pi_j(\U_0(M_2(A)))$. Now, from \cref{equal} and \cref{equality},
\[
 N_1[(\psi_K)_2\circ \iota^B\circ g] = N_2[(\psi_K)_2\circ\phi_2\circ u] =  N_2[\psi_F\circ\phi_2\circ u] = N_2[u].
\]
Since $(\psi_K)_2\circ\iota^B\circ g=\iota^A\circ\psi_K\circ g$, we have
\[
N_1[\iota^A\circ\psi_K\circ g]=N_2[u].
\]
Therefore, if $\omega :=\psi_K\circ g$ and $q := \frac{N_1}{N_2m}$, then
\[
(\iota^A)_{\ast}\otimes\text{id}\left([\omega ]\otimes q\right)=[u]\otimes\frac{1}{m}
\]
proving that $(\iota^A)_{\ast}\otimes \text{id}$ is surjective.
\end{proof}

We are now in a position to complete the proof of \cref{main_thm: rokhlin}.

\begin{cor}
Let $\alpha:G\to \text{Aut}(A)$ be an action of a compact Lie group on a separable C*-algebra $A$ such that $\dim_{Rok}^c(\alpha) < \infty$. If $A$ is rationally $K$-stable ($K$-stable), then so is $A\rtimes_{\alpha} G$.
\end{cor}
\begin{proof}
We first discuss the case of $K$-stability: Let $X$ be the (finite dimensional) metric space space obtained from \cref{thm: seq_split_rokhlin_dimension}. By \cref{prop: target_field}, $C(X,A)\rtimes_{\gamma} G$ is a continuous $C(X/G)$-algebra, each of whose fibers are isomorphic to $A\otimes \mathcal{K}(L^2(G))$, and are hence $K$-stable. Since $X$ is compact and metrizable, so is $X/G$. Furthermore, since $G$ is a compact Lie group, it follows that
\[
\dim(X/G) \leq \dim(X) < \infty
\]
by \cite[Corollary 1.7.32]{palais_g_space}. By \cite[Theorem A]{vaidyanathan}, we conclude that $C(X,A)\rtimes_{\gamma} G$ is $K$-stable, and hence $A\rtimes_{\alpha} G$ is $K$-stable by \cref{prop: sequentially_split}. \\

The argument for rational $K$-stability is entirely similar, except that we apply \cref{main_thm} instead of \cite[Theorem A]{vaidyanathan}.
\end{proof}
\section*{Acknowledgements}
The first named author is supported by UGC Junior Research Fellowship No. 1229, and the second named author was partially supported by the SERB (Grant No. MTR/2020/000385).

\end{document}